\documentclass[reqno,11pt]{amsart}
\usepackage{amsmath,amssymb,amsthm, verbatim}
\usepackage{url}
\usepackage[usenames, dvipsnames]{color}

 \numberwithin{dummy}{section}

\usepackage[letterpaper,hmargin=1in,vmargin=1in]{geometry}

\usepackage{soul}
\usepackage{graphicx}
\usepackage{enumerate,pinlabel}
\usepackage{mathrsfs,graphicx,url}
\usepackage[usenames,dvipsnames]{xcolor}
\usepackage[colorlinks=true,linkcolor=Black,citecolor=Black, urlcolor=Black]{hyperref}
\renewcommand{\Re}{\operatorname{Re}}
\renewcommand{\Im}{\operatorname{Im}}

\numberwithin{equation}{section}
\newcommand{\mc}[1]{\mathcal{#1}}

\theoremstyle{plain}
\newtheorem{theorem}{Theorem}
\newtheorem*{theorem*}{Theorem}
\newtheorem*{lemma*}{Lemma}
\newtheorem{lemma}{Lemma}[section]
\newtheorem*{conjecture*}{Conjecture}

\newtheorem{proposition}{Proposition}[section]

\theoremstyle{definition}

\theoremstyle{remark}

\newcommand{\ep}{\varepsilon}
\newcommand{\NN}{\mathbb{N}}
\newcommand{\R}{\mathbb{R}}
\newcommand{\US}{\mathbb{S}}
\newcommand{\C}{\mathbb{C}}

\newcommand{\dom}{\mathcal{D}}

\newcommand\esssup{\mathop{\rm ess\,sup}}
\newcommand\esssupp{\mathop{\rm ess\,supp}}
\newcommand\supp{\mathop{\rm supp}}

\newcommand\real{\mathop{\rm Re}}
\newcommand\imag{\mathop{\rm Im}}
\newcommand\Ai{\textrm{Ai}}
\newcommand\Bi{\textrm{Bi}}
\newcommand*{\defeq}{\mathrel{\vcenter{\baselineskip0.5ex \lineskiplimit0pt

                     \hbox{\scriptsize.}\hbox{\scriptsize.}}}%
                     =}
\newcommand{\jg}[1]{{\textcolor{purple}{#1}}}
\newcommand{\e}{\epsilon}
\newcommand{\loc}{\operatorname{loc}}
\newcommand{\comp}{\operatorname{comp}}
                     
                     \title{Semiclassical resolvent bounds for compactly supported radial potentials}

\author{Kiril Datchev}
\address{Department of Mathematics, Purdue University,  West Lafayette, IN, 47907-2067, USA}
\email{kdatchev@purdue.edu}

\author{Jeffrey Galkowski}
\address{Department of Mathematics, University College London, London, WC1H 0AY, UK}
\email{j.galkowski@ucl.ac.uk}

\author{Jacob Shapiro}
\address{Department of Mathematics, University of Dayton, Dayton, OH 45469-2316, USA}
\email{jshapiro1@udayton.edu}

\begin{document}
\begin{abstract}
We employ separation of variables to prove weighted resolvent estimates for the semiclassical Schr\"odinger operator $-h^2 \Delta + V(|x|) - E$ in dimension $n \ge 2$, where $h, \, E > 0$, and $V: [0, \infty) \to \R$ is $L^\infty$ and compactly supported. The weighted resolvent norm grows no faster than $\exp(Ch^{-1})$, while an exterior weighted norm grows $\sim h^{-1}$. We introduce a new method based on the Mellin transform to handle the two-dimensional case. 

\end{abstract}
\maketitle

\section{Introduction and statement of results}
Let $\Delta \defeq \sum_{j=1}^n \partial^2_{x_j}$ be the Laplacian on $\mathbb{R}^n$, $n \ge 2$. We consider the semiclassical Schr\"odinger operator on $L^2(\mathbb R^n)$ given by
\begin{equation*}
Pu = P(h)u \defeq -h^2 \Delta u + V(|x|)u,
\end{equation*}
where $V\colon [0,\infty) \to \mathbb R$ is $L^\infty$ and compactly supported, and $h>0$ is a semiclassical parameter. Then $P \colon H^2(\mathbb{R}^n) \to L^2(\mathbb R^n)$ is self-adjoint, and the resolvent $(P - z)^{-1}$ is bounded  on $L^2(\mathbb{R}^n)$ for all $z \in \mathbb{C} \setminus \mathbb{R}$. Throughout the article, we let $r\defeq |x|$, and  $\langle x \rangle \defeq (1 + |x|^2)^{1/2}$.

Our first result is an exponential upper bound on the limiting absorption resolvent.
\begin{theorem} \label{exp thm}
Let $n \ge 2$. Fix $[E_{\emph{min}}, E_{\emph{max}}] \subseteq (0, \infty) $ and $1/2 < s \le 1$. There exist  $C, \, h_0 > 0$, such that
\begin{equation} \label{radial exp est}
\|\langle x \rangle^{-s}  (P - E - i \ep)^{-1} \langle x \rangle^{-s}  \|_{L^2(\R^n) \to L^2(\R^n) } \le e^{C/h},
\end{equation}
for all $\ep > 0$, $h \in (0,h_0]$, and $E\in [E_{\emph{min}}, E_{\emph{max}}]$. 
\end{theorem}

In addition, we prove a `non-trapping' type estimate for the resolvent in the exterior of a large ball. Let 
$$
R_0 = R_0(V) \defeq \sup\{ r\in [0,\infty)\,:\, r\in \esssupp V\},
$$ 
and define
\begin{equation} \label{M 0 condition}
\begin{split}
M_0 &= M_0(V,E) \defeq \\
 &\inf \{m >0 \mid  V(r) + m r^{-2} - E \ge 0, \text{ for almost all }  r \text{ in a neighborhood of $(0,R_0(V)]$}\}.
\end{split}
\end{equation}
For example, if $V$ is the characteristic function of the interval $(0, R_0]$, then $M_0 = ER^2_0$.  If $V$ is continuous at $R_0$, $M_0 =  \esssup_{[0,R_0]}r^2(E - V)$. Note that always $M_0 \ge E R_0^2$ because we require $M_0r^{-2}-E \ge 0$ for some $r > R_0$. 
Finally, put 
\begin{equation} 
\label{e:r1} R_1 = R_1(V, E) \defeq \sqrt{M_0(V,E)/E}.
\end{equation}
\begin{theorem} \label{ext thm}
Let $n \ge 2$. Fix $[E_{\emph{min}}, E_{\emph{max}}] \subseteq (0, \infty) $, $1/2 < s \le 1$ and $R> \sup_{E\in [E_{\min},E_{\max}]} R_1(V, E)$. There exist $C, \, h_0 > 0$, such that
\begin{equation} \label{radial ext est}
\|\langle x \rangle^{-s} \mathbf{1}_{\ge R} (P - E - i \ep)^{-1}\mathbf{1}_{\ge R} \langle x \rangle^{-s}  \|_{L^2(\R^n)  \to L^2(\R^n) } \le \frac{C}{h},
\end{equation}
for all $\ep > 0$, $h \in (0,h_0]$, and $E\in [E_{\emph{min}}, E_{\emph{max}}]$, where $\mathbf{1}_{\ge R}$ is the characteristic function of $\{x \in \mathbb R^n \colon |x| \ge R\}$. 
\end{theorem}

Theorem \ref{ext thm} is optimal in the sense that \cite[Theorem 3]{daji20} shows that \eqref{radial ext est} is false in general when $R<\sqrt{M_0/E}$. For example, if $V \in C_0^\infty([0,1))$, $E < \max (-r^2 V) $, and $R<\sqrt{M_0/E}$, then (by \cite[(2.9) and (4.10)]{daji20}) the left hand side of \eqref{radial ext est} is bounded below by $e^{1/Ch}$ for $h$ tending to zero along a sequence of positive values.

The novelty of Theorems \ref{exp thm}  and \ref{ext thm} is that they bound the weighted resolvent for an arbitrary compactly supported, radial $L^\infty$ potential. The $h$-dependencies on the right sides of \eqref{radial ext est} and \eqref{radial exp est} are sharp in general, see \cite{ddz15} and \cite{daji20} for exponential lower bounds, and recall that the free resolvent ($V \equiv 0$) has a $Ch^{-1}$ lower bound (to see this, consider $u=e^{i\sqrt{E} h^{-1} x_1}\chi(x)$ for some $ 0 \not \equiv \chi \in C_0^\infty(\mathbb{R}^n)$). 

Vodev's work \cite{vo21} shows, for dimension $n \ge 3$, a bound like \eqref{radial exp est} still holds for radial potentials decaying like $\langle r \rangle^{-\delta}$, $\delta > 2$, except with the right side replaced by $e^{Ch^{-4/3}}$ (bounds with additional losses hold for $V$ decaying more slowly). For $V \in L^\infty(\R^n;\R)$ not necessarily radial, $n \ge 2$, with $V  = O(\langle r \rangle^{-\delta})$, $\delta > 2$, the best known weighted resolvent upper bound is $e^{Ch^{-4/3}\log(h^{-1})}$ \cite{gash20}. In dimension $n \ge 2$, it is an open problem to determine the optimal $h$-dependence of the resolvent for $V \in L^\infty$. In contrast, when $n = 1$, an $e^{Ch^{-1}}$ bound holds even if $V \in L^1(\R; \R)$ \cite{dash20}. As far as the authors are aware, Theorem~\ref{ext thm} is the first exterior estimate for any class of $L^\infty$ potentials in dimension higher than one.

Proofs of semiclassical resolvent estimates have a long history and are an active research topic. Burq \cite{bu98} was the first to show an $e^{Ch^{-1}}$ bound for smooth perturbations of the Laplacian on $\R^n$. Several extensions followed \cite{vo00, bu02, sj02, cavo02}. The exterior bound \eqref{radial ext est} was first proved by Cardoso and Vodev \cite{cavo02}, refining a preliminary estimate of Burq \cite{bu02}. More recent works on resolvent estimates in lower regularity include \cite{da14, vo14, rota15, klvo19, sh19, vo19, gash20, gash21, sh20, vo20a, vo20b, vo20c, vo21}.

Stronger bounds on the resolvent are known when $V \in C^\infty_0(\R^n ; \R)$ and conditions are imposed on the classical flow $\Phi(t) = \exp t(2 \xi \partial_x - \partial_x V(x) \partial_\xi)$ (note that $\Phi(t)$ may be undefined in our case). The key dynamical object is the \textit{trapped set} $\mathcal{K}(E)$ at energy $E > 0$, defined as the set of $(x, \xi) \in T^*\R^n$ such that $|\xi|^2 + V(x) = E$ and $|\Phi(t)(x, \xi)|$ is bounded as $|t| \to \infty$. If $\mathcal{K}(E) = \emptyset,$ that is, if $E$ is \textit{nontrapping}, Robert and Tamura \cite{rota87} showed the weighted resolvent is bounded by $Ch^{-1}$. We may thus think of \eqref{radial ext est} as a low regularity analog; it says that applying cutoffs supported sufficiently far from zero removes the losses from \eqref{radial exp est} due to trapping.

Theorem \ref{exp thm} allows one to obtain, at high frequency, bounds on and resonance free regions for, the meromorphic continuation of the cutoff resolvent of the operator $-c^2(|x|) \Delta$ on $L^2(\R^n, c^{-2}dx)$. Here $c \in L^\infty([0,\infty); (0,\infty))$ is called the \textit{wavespeed} and satisfies 
\begin{gather}
\supp(1 - c) \text{ is compact}, \label{one minus c}  \\
c_0 < c(r) < c_1 \text{ for some $c_0, c_1 > 0$ and for all $r \in [0,\infty)$.} \label{bd abv blw}
\end{gather}
More precisely, one obtains, 
\begin{theorem} \label{res free}
Let $n \ge 2$, and suppose $c \in L^\infty([0,\infty); (0,\infty))$ obeys \eqref{one minus c} and \eqref{bd abv blw}. For each $\chi \in C^\infty_0(\R^n)$, there exist constants $C_1, C_2, M > 0$ such that the cutoff resolvent \\$\chi R(\lambda) \chi \defeq \chi (-c^2\Delta - \lambda^2)^{-1} \chi$ continues analytically from $\imag \lambda > 0$ into the set\\$\{ \lambda \in \mathbb{C} : |\real \lambda | > M, \, \imag \lambda > -e^{C_2|\real \lambda|} \}$, where it satisfies the bound
\begin{equation}
\| \chi R(\lambda) \chi \|_{L^2(\R^n ,\, c^{-2}dx) \to L^2{(\R^n ,\,c^{-2}dx)}} \le e^{C_1 |\real \lambda|}. \label{exp high energy est}
\end{equation}
\end{theorem}
The proof of Theorem \ref{res free} is the same as the proof of \cite[Proposition 5.1]{sh18}, and is seen by identifying $V = 1 - c^{-2}$, $h = |\real \lambda|^{-1}$ and applying \eqref{radial exp est}.
 
Theorem \ref{res free} implies logarithmic local energy decay for the wave equation
\begin{equation} \label{wave equation}
\begin{cases}
(\partial_t^2 - c^2(x)\Delta) u(x,t) = 0, & (x,t) \in \R^n  \times (0, \infty), \, n \ge 2, \\
 u(x,0) = u_0(x) \in H^2(\R^n),\\
 \partial_t u(x,0) = u_1(x) \in H^1(\R^n),
\end{cases}
\end{equation}
where the initial data are compactly supported. Such a decay rate was first proved by Burq \cite{bu98} for $c$ smooth, and allowing for a smooth Dirichlet obstacle.

The assumption that $c =1$ outside of a compact set is necessary not only to establish \eqref{exp high energy est}, but also to study the low-frequency behavior of the cutoff resolvent. Under this assumption, one can see that this behavior is exactly the same as the case $c \equiv 1$ \cite[Proposition 4.1]{sh18}, which in turn is well-known (see, e.g., \cite[Section 1.1]{vo01}). With both the low and high frequency behavior of the $\chi R(\lambda) \chi$ illuminated, a logarithmic local energy decay rate for the solution of \eqref{wave equation} follows exactly as in \cite[Sections 6 and 7]{sh18}.

\subsection{Ideas from the proofs}
When $V : \R^n \to \R$ has limited regularity, proofs of resolvent estimates typically proceed by a modified positive commutator strategy, see e.g., \cite[Section 2]{gash20}. The potential is treated as a perturbation because it cannot be differentiated in some or all directions.  The issue one needs to overcome is that, in the positive commutator scheme, the potential appears in a term without an apparent sign (even though $V$ is real-valued), which in turn is difficult to control as $h \to 0^+$. Controlling this term results in an estimate from above by $e^{Ch^{-4/3}\log(h^{-1})}$.

Due to the work of Meshkov~\cite{me92}  on the Landis conjecture, one knows that the exponent $h^{-4/3}$ is optimal for compactly supported complex valued potentials (see Appendix~\ref{s:mesh} for an explanation). Therefore, any improvement upon the exponent $h^{-4/3}$ in the resolvent estimate for an $L^\infty$ compactly supported potential must involve additional assumptions on the potential $V$ (e.g. reality/radiality).

The radial symmetry of the potential is used in two ways. First, it allows us to study the high angular momenta separately from the low angular momenta. In particular, we use a spherical energy type estimate to obtain resolvent estimates for low frequencies. Second, decomposition by angular momentum enables us to take advantage of reality of $V$ using ODE techniques. In particular, for large enough angular momenta $m_j$, reality of $V$ yields useful monotonicity properties of certain solutions $u_0$ and $u_1$ to $(-h^2 \partial^2_r + V + m_jr^{-2} - E)u = 0$ (see their construction in Section \ref{reduction section}). Control of $u_0$ and $u_1$ gives a bound on the integral kernel of $(-h^2 \partial^2_r + V + m r^{-2} - E - i0)^{-1}$, which together with WKB and Bessel function asymptotics, yields the sharp semiclassical resolvent estimates.

We remark that, in the setting of Theorem \ref{exp thm}, our ODE methods do not yield estimates on $u_0$ and $u_1$ for small $m_j$. However, we are able to prove good enough resolvent estimates for these frequencies.  
In doing so, we prove the following proposition.
\begin{proposition}
\label{p:1DResolve}
Fix $[E_{\min},E_{\max}]\subseteq (0,\infty)$ and $1/2 < s \le 1$. There exist $C, \, h_0>0$ such that for $0\leq \ep \leq 1$, $h \in (0, h_0]$, $m\geq -h^2/4$, and $E\in [E_{\min},E_{\max}]$,
\begin{equation}
\label{vod exp}
\|\langle r \rangle^{-s}(-h^2 \partial^2_r + V + m r^{-2} - E - i\ep)^{-1}\langle r \rangle^{-s} \|_{L^2(0,\infty) \to L^2(0, \infty)}\leq e^{C(1 + |m|^{1/2} )/h}.
\end{equation}
\end{proposition}
Proposition~\ref{p:1DResolve} is motivated by~\cite[Proposition 3.1]{vo21} where a similar estimate is proved for $m\geq 0$. However, in order to handle the case of $n=2$, we need to extend these resolvent estimates to $m\geq -h^2/4$. The proof of Proposition~\ref{p:1DResolve} utilizes methods inspired by the b-calculus from microlocal analysis to estimate $u$ by $(-h^2\partial_r^2 +V+mr^{-2}-E)u$ near $0$, and then employs a spherical energy method to handle the region away from zero. 

We expect Theorems \ref{exp thm} and \ref{ext thm} still hold for potentials $V$ which are radial, real and non-compactly supported, with sufficient decay toward infinity. A difficulty with treating this case is finding a suitable replacement for the WKB and Bessel function asymptotics we use in Section \ref{ext est section}. A development in this direction is that, since we posted this article as a preprint, Vodev \cite{vo22} designed a new Carleman estimate for $n \ge 3$, large angular momenta, and radial $V = O(\langle r \rangle)^{-\delta}$, $\delta > 4$; it yields improved resolvent bounds compared to \cite{vo21}, and an alternative proof of \eqref{radial exp est}. 

The organization of the paper is as follows. In Section \ref{reduction section} we give an overview of the plan to prove Theorems \ref{exp thm} and \ref{ext thm} via separation of variables. In Section \ref{ext est section} we use WKB and Bessel function asymptotics to prove Theorem \ref{ext thm}. In Section \ref{s:low} we use the Mellin transform and energy estimates to prove an exponential estimate which is optimal for low angular momenta but not for high ones. In Section \ref{s:exp thm} we use ODE analysis to remove the losses for high angular momenta and complete the proof of Theorem \ref{exp thm}. 

\medskip
\noindent{\textsc{Acknowledgements:}} We thank the anonymous referee, whose helpful comments improved the exposition. The authors gratefully acknowledge from the following sources: KD from NSF DMS-1708511, J.G. from EPSRC Early Career Fellowship EP/V001760/1, and J.S. from NSF DMS-2204322 and a University of Dayton Research Council Seed Grant. On behalf of all authors, the corresponding author states that there is no conflict of interest.\\

\noindent{\textsc{Data Availability Statement:}} Data sharing not applicable to this article as no datasets were generated or analyzed during the current study.

\section{Reduction to a family of one dimensional resovlent estimates} \label{reduction section}
In this section, we use separation of variables to reduce the proofs of Theorems \ref{exp thm} and \ref{ext thm} to a family of one dimensional resolvent estimates. We begin by recalling the conjugation
\begin{equation*}
r^{(n-1)/2} (-\Delta) r^{-(n-1)/2} = -\partial^2_r - \frac{\Delta_{\US^{n-1}}}{r^2} +  \frac{(n-1)(n-3)}{4r^2}.
\end{equation*}
 We then put
\[
m_j = h^2(\sigma_j + 4^{-1}(n-1)(n-3)),
\]
where $0 = \sigma_0 < \sigma_1 = \sigma_2 \le \sigma_3 \le \cdots$ are the eigenvalues of the nonnegative Laplace-Beltrami operator on the unit sphere $\mathbb S^{n-1} \subseteq \R^n$, repeated according to multiplicity. (Recall that the eigenvalues of the unit sphere are $k^2 + (n-2)k$, $k = 0, \, 1, \, \dots$.) Denote by $\mathbf{Y}_0, \ \mathbf{Y}_1, \dots$ a corresponding sequence of orthonormal real eigenfunctions. Also define
\begin{equation*}
P_m \defeq -h^2 \partial^2_r + V(r) + mr^{-2} , \qquad m \ge -\frac{h^2}{4}.
\end{equation*}
The operator $P_m$, acting on $L^2(\R_+)$, $\R_+ \defeq (0,\infty)$, with domain $ C^\infty_0(\R_+)$, is symmetric. As $P_m$ commutes with complex conjugation, by von Neumann's theorem it has equal defect indices, and thus has self adjoint extensions. Elements in the domain $\dom(P^*_m)$ are precisely those functions $u \in L^2(\R_+)$ having $u$ and $\partial_r u$ absolutely continuous along with $P_m u \in L^2(\R_+)$. Moreover, since $\|r^{-1} \varphi \|_{L^2(\R_+)} \le 2 \| \partial_r \varphi \|_{L^2(\R_+)}$ for any $\varphi \in C^\infty_0(\R_+)$, $P_m$ is a semibounded operator, thus a particular self-adjoint extension of $(P_m, C^\infty_0(\R_+))$ is its Friedrichs extension (see \cite[Theorem X.23]{reedsimon2}), which we denote by $(P_m, \dom_m)$; this is the self-adjoint extension of $(P_m, C^\infty_0(\R_+))$ that we work with throughout the paper.

We start by studying the resolvent kernel for $P_m$. To do this, we construct two convenient linearly independent solutions, $u_0$ and $u_1$, which are not in $\dom_m$, to 
\begin{equation} \label{u zero u one ep}
P_m u_{j} = (E+ i \varepsilon)u_{j}, \qquad j \in \{0,1\}.
\end{equation}
To define these solutions, let
\begin{equation*}
\begin{gathered}
 \varphi_J(r) = r^{1/2} J_\nu(\lambda r), \qquad   \varphi_Y(r) = r^{1/2} Y_\nu(\lambda r), \\
m \ge -\frac{h^2}{4}, \quad \nu = h^{-1}\big(m+\frac{h^2}{4} \big)^{1/2} \ge 0, \quad \lambda = \frac{\sqrt{E+i\varepsilon}}{h}, \quad E \in [E_{\min}, E_{\max}], \quad \ep \ge 0,
 \end{gathered}
\end{equation*}
where $J_\nu(\cdot)$ and $Y_\nu(\cdot)$ are the Bessel functions of the first and second kind, respectively \cite[Section 10.2]{dlmf}. Note that 
$$
(-h^2\partial_r+mr^{-2}-E-i\ep)\varphi_{J/Y}=0.
$$
Next, define $\varphi_0 = \varphi_J$ and, inductively,
\[
 \varphi_{n+1}(r) =  \frac{\pi}{2h^2} \int_0^r (\varphi_Y(r)\varphi_J(r') - \varphi_J(r)\varphi_Y(r')) V(r')\varphi_n(r')dr',\qquad n\geq 0.
\]
Finally, put 
\begin{equation} \label{defn u0}
 u_0(r) := \sum_{n=0}^\infty \varphi_n(r).
\end{equation}
In Appendix \ref{series construction u 0}, we prove

\begin{lemma}
\label{l:u0}
The series~\eqref{defn u0} and its first derivative converge uniformly for $(r, E, \ep)$  in compact subsets of $\R_+ \times [E_{\min}, E_{\max}] \times [0, \infty)$. In addition, $u_0(r) > 0$ near $r = 0$,
\begin{gather}
 u_0(r) = \varphi_0(r) +\frac{\pi}{2h^2} \int_0^r (\varphi_Y(r)\varphi_J(r') - \varphi_J(r)\varphi_Y(r'))V(r')u_0(r')dr', \label{u0 integral eqn} \\
( -h^2 \partial_r^2 + V(r)  + m r^{-2}  - E - i \ep)u_0 = 0, \label{u0 solves eqn}\\
 \label{uAsymptotic}
 \lim_{r\to 0^+}\big(\tfrac{n-1}{2r} u_0(r) - u_0'(r)\big)r^{(n-1)/2}= 0,
\end{gather}
and for all $ r^* >0, \nu \geq 0$, there is $C_{\nu,r^*}$ such that 
\begin{equation} \label{u zero near zero}
|u_0(r)|\leq C_{\nu,r^*}r^{\nu+\frac{1}{2}},\qquad r\in (0,r^*],
\end{equation}
\end{lemma}

Next, put
\begin{equation} \label{defn u1}
 u_1(r) = \varphi_J(r) + i \varphi_Y(r), \qquad r >R_0,
\end{equation}
and extend $u_1$ by requiring that it solve \eqref{u zero u one ep}. By \eqref{defn u1} and \cite[Theorem 2.1 in Section 5.2.1]{olver},  $u_1$ depends continuously on $(r, E, \ep)$ varying in $\R_+ \times [E_{\text{min}}, E_{\text{max}}] \times [0, \infty)$.
Also, from  \cite[10.17.5]{dlmf}, for all $r^* \geq 1$, $\nu\geq 0$, there is $C_{r^*,\nu}>0$ such that 
\begin{equation} \label{u one near infty}
|u_1(r)| \leq C_{r^*,\nu} e^{- \imag \lambda r}, \qquad r \geq r^*.
\end{equation}

\begin{lemma}
The functions $u_0$ in~\eqref{defn u0} and $u_1$ in~\eqref{defn u1} are linearly independent for all $\ep \geq 0$.
\end{lemma}
\begin{proof}
First consider $\ep>0$ and suppose $u_1$ and $u_0$ are linearly dependent. Then, since \\ $u_0\in L^2((0,1])$ by~\eqref{u zero near zero} and $u_1\in L^2([1,\infty))$ by~\eqref{u one near infty}, we would have $u_1\in L^2([0,\infty))$. In particular, $u_1$ would be an $L^2$-solution of $(P_m - E - i \ep)u =0$ and thus must vanish identically. Next, when $\ep = 0$, then by the Bessel function equation \eqref{e:wdifeq},  $u_0(r) = A \varphi_J(\lambda r) + B\varphi_Y(\lambda r)$ when $r > R_0$ and for some real constants $A$ and $B$. Comparing with \eqref{defn u1} and using the linear independence of $J_\nu(z)$ and $Y_\nu(z)$, $z > 0$, we conclude that $u_0$ and  $u_1$ are linearly independent when $\ep = 0$, too.

\end{proof}

We can now define the resolvent kernel for $P_m$,
\begin{equation} \label{resolv kernel ep}
\begin{gathered}
K(r,r') = K(r,r'; m, E, \varepsilon, h) = -\frac{u_0(r) u_1(r')}{h^2 W}, \\
 r \le r', \, m \ge -\frac{h^2}{4}, \, E \in [E_{\text{min}}, E_{\text{max}}], \, \ep \ge 0, 
\end{gathered}
\end{equation}
and for $r'< r$, $K(r,r') = K(r',r)$, where
$W = u_0u'_1 - u'_0u_1$ is the Wronskian of $u_0$ and $u_1$.

\begin{lemma}
For $E, \varepsilon>0$, $u \in L^2(\R^n)$ and $v \in L^2(\R_+)$,
\begin{gather}
\begin{split}
(P_{m} - E - i\ep)^{-1} v &= \int^\infty_0 K(r,r'; m, E, \varepsilon, h) v(r')dr'.  \label{relate oned resolv to kernel}\\
&=  -h^{-2} W^{-1} \big( u_1(r) \int^r_0  u_0(r') v(r') dr' +  u_0(r) \int_r^\infty  u_1(r') v(r') dr'  \big),
\end{split}\\
 (P-E-i\varepsilon)^{-1} u = \sum_{j=0}^\infty \mathbf{Y}_j \int_0^\infty \int_{\mathbb S^{n-1}} r^{-(n-1)/2} K(r,r'; m_j, E, \varepsilon, h) (r')^{(n-1)/2} u(r',\theta)\mathbf{Y}_j(\theta)d\theta dr'. \label{decompose resolv}
 \end{gather}
\end{lemma}
\begin{proof}
To prove \eqref{relate oned resolv to kernel}, it suffices to work with $v \in C^\infty_0(\R_+)$. We check that the right side of \eqref{relate oned resolv to kernel} belongs to $\dom_m$ and, that applying $P_m - E -i\ep$ yields $v$. The latter is a direct computation, while the former follows from \eqref{u zero near zero}, \eqref{u one near infty}, and the fact that a characterization of $\dom_m$  is
\begin{equation} \label{nz characterize}
\dom_m = \begin{cases}  \big\{ f \in L^2(\R_+) : P_m f \in L^2(\R_+), \,  r^{-\nu - \tfrac{1}{2}} f \in L^\infty \text{ near $r = 0$} \big\} & \tfrac{-h^2}{4} \le m <  \tfrac{3h^2}{4}, \\
 \big\{ f \in L^2(\R_+) : P_m f \in L^2(\R_+) \big\} & m \ge \tfrac{3h^2}{4}.
  \end{cases}
\end{equation}
 See \cite[Section 6, Example I]{nize92}.
 
We need only verify \eqref{decompose resolv} for $u = \mathbf{Y}_j v$ with $v \in C_0^\infty(\R_+)$, as such functions have dense linear span in $L^2(\R^n)$. In this case, the right side of \eqref{decompose resolv} reduces to
\begin{equation} \label{decompose resolv single freq}
 \mathbf{Y}_j \int_0^\infty r^{-(n-1)/2} K(r,r'; m_j, E, \varepsilon, h) (r')^{(n-1)/2} v(r')dr'.
\end{equation}
To show \eqref{decompose resolv single freq} and $ (P-E-i\varepsilon)^{-1} \mathbf{Y}_jv$ coincide, we check that applying  $P-E-i\varepsilon$, respectively $\Delta$, to \eqref{decompose resolv single freq} in the sense of distributions, results in $\mathbf{Y}_jv$, respectively some function in $L^2(\R^n)$. Both computations are handled by integrating by parts in polar coordinates, on domains of the form $\{ x \in \R^n : |x| > \delta>0\}$, and sending $\delta \to 0$.  All boundary terms that appear in the calculation vanish as $\delta \to 0$, thanks to~\eqref{uAsymptotic} and~\eqref{u zero near zero}. We leave the remaining details to the reader.  \\
\end{proof}

We next consider the limit as $\ep \to 0^+$. Recall that \cite[Theorem 4.2]{ag75} for any $s > 1/2$, and $E \in [E_{\min}, E_{\max}]$ , the limit
\begin{equation} \label{limit absorpt}
 \langle x \rangle^{-s}(P - E - i0)^{-1} \langle x \rangle^{-s} \defeq  \lim_{\ep \to 0^+} \langle x \rangle^{-s} (P - E - i\ep)^{-1} \langle x \rangle^{-s}
\end{equation}
 exists in the uniform topology $L^2(\R^n) \to L^2(\R^n)$. By \eqref{decompose resolv}, \eqref{relate oned resolv to kernel}, and because $K(r, r'; m_j, E, \ep, h)$ depends continuously on $(r,r', E, \ep)$ varying in compact subsets of $\R_+ \times \R_+ \times [E_{\text{min}}, E_{\text{max}}] \times [0, \infty)$, for each $j$ and $u, v \in C^\infty_0(\R_+)$, 
 \begin{equation*}
 \begin{split} 
  \langle \mathbf{Y}_j u, \langle &x \rangle^{-s}(P - E - i0)^{-1} \langle x \rangle^{-s} \mathbf{Y}_j v \rangle_{L^2} \\
 &= \lim_{\ep \to 0^+}  \langle r^{(n-1)/2} u,  \langle r \rangle^{-s} \int_0^\infty K(r,r'; m_j, E, \varepsilon, h)  \langle r' \rangle^{-s}  (r')^{(n-1)/2} v(r') dr' \rangle_{L^2(\R_+)} \\
 &=  \lim_{\ep \to 0^+}  \langle r^{(n-1)/2} u,  \langle r \rangle^{-s}(P_{m_j} - E - i\ep)^{-1}  \langle r' \rangle^{-s}  (r')^{(n-1)/2} v(r') \rangle_{L^2(\R_+)} \\
  &=  \langle r^{(n-1)/2} u,  \langle r \rangle^{-s} \int_0^\infty K(r,r'; m_j, E, 0, h)  \langle r' \rangle^{-s}  (r')^{(n-1)/2} v(r') dr' \rangle_{L^2(\R_+)}. 
 \end{split}
\end{equation*}
Thus, as functions $\{\mathbf{Y}_j v : j \in \NN_0, \, v \in C^\infty_0(\R_+) \}$ have dense linear span in $L^2(\R^n)$, to bound the norm of \eqref{limit absorpt} on $L^2(\R^n)$, it suffices to bound, uniformly in $j$, the kernel $K(r,r'; m_j, E, 0, h)$, or alternatively, $\|\langle r \rangle^{-s}(P_{m_j} - E - i\ep)^{-1}  \langle r \rangle^{-s}\|_{L^2(\R_+) \to L^2(\R_+)}$ for $\ep > 0$ small.

If the bounds \eqref{radial exp est} and \eqref{radial ext est} are established for $\ep = 0$, it is well-known that one can use resolvent identities to show they hold for $\ep > 0$ as well. We omit the proof of this but refer the reader to the relevant results, see \cite[Proposition 3]{brpe00} and \cite[Theorem 1.5]{vo14}.

\section{Exterior estimates} \label{ext est section}

Theorem \ref{ext thm} follows from the stronger statement that, for any  $R > \sup_{E\in [E_{\min},E_{\max}]} R_1(V, E)$ (see \eqref{e:r1}), we have 
\begin{equation} \label{bound cutoff kernel}
|K(r, r')| = |K(r,r',m,E,0,h)| \le \frac{C}{h},
\end{equation}
uniformly for $r$ and $r'$ in $[R,\infty)$ obeying $r \le r'$, $m\geq -h^2/4$, $h$ small, and $E\in [E_{\max},E_{\min}]$. We prove \eqref{bound cutoff kernel} in this section, over the course of two lemmas. The first lemma establishes \eqref{bound cutoff kernel} for $m \le M_+$, where 
\begin{equation} \label{defn M plus}
M_+ = M_+(E) \defeq M_0 + \frac{E(R^2 - R^2_1)}{2}.
\end{equation}
 In the second lemma, we prove \eqref{bound cutoff kernel} for $m > M_+$.

\begin{lemma} \label{Olver lemma}
There are $C$ and $h_0$ such that \eqref{bound cutoff kernel} holds for all $r$ and $r'$ in $[R,\infty)$ obeying $r \le r'$, $h \in(0,h_0]$, $E\in[E_{\min},E_{\max}]$ and for all $m \in [-h^2/4,M_+]$.
\end{lemma}

\begin{proof}
The proof is essentially the same as the one in~\cite[Lemma 1]{daji20}. It is based on WKB or Liouville-Green asymptotics for solutions to an ODE with real coefficients in a region with no turning point. First, observe that if $r \ge R $ and $m \le M_+$, then
\begin{equation*}
mr^{-2} \le \frac{ 2M_0 +  E(R^2 - R^2_1)}{2r^2} \le \frac{ 2ER^2_1+  E(R^2 - R^2_1)}{2R^2} = E -  \frac{E(R^2 - R^2_1)}{2R^2}.
\end{equation*}
Thus $E -mr^{-2}$ is bounded below by a positive constant uniformly in $E$,  $r$ and $m$.

By \cite[Section 6.2.4]{olver}, there are real numbers $A = A(h)$ and $B =B(h)$ such that, for $r \ge R$ we have
\[
 u_0(r) = \frac 1 {\sqrt[4]{E- mr^{-2}}} \Big(\sum_\pm (A\pm iB) \exp\Big(\pm \frac i h \int_R^r \sqrt{E- ms^{-2}}ds \Big)(1+\varepsilon_\pm(r))\Big),
\]
where $\varepsilon_+$ and $\varepsilon_-$ satisfy
\[
 |\varepsilon_\pm(r)| + h|\varepsilon_\pm'(r)| \lesssim \frac{h}{r^3}.
\]
By \cite[Section 6.2.4]{olver}, there is a constant $D = D(h)$ so that for $r \ge R$,
\[
 u_1(r) = \frac {D}{\sqrt[4]{E- mr^{-2}}}\exp\Big(\frac i h \int_R^r \sqrt{E- ms^{-2}}ds \Big)(1+\varepsilon_+(r)),
\]
where we rule out the presence of an $\exp\Big(- \frac i h \int_R^r \sqrt{E- ms^{-2}}ds \Big)$ term by using large-argument Bessel function asymptotics \cite[10.17.5 and 10.17.11]{dlmf}
to show that 
\begin{equation*}
\partial_r u_1 - i \frac{\sqrt{E}}{h} u_1 \to 0, \qquad \text{as $r \to \infty$}.
\end{equation*}

 Next we compute the Wronskian
\[
 W = \frac D {\sqrt{E- mr^{-2}}}\Big(A-iB\Big) \frac {2i}h \sqrt{E- mr^{-2}}\Big(1 + O(hr^{-3})\Big) = \frac{2D(B+iA)}h,
\]
where we dropped the remainder because $W$ is independent of $r$. Plugging these formulas for $u_0$, $u_1$, and $W$ into the formula \eqref{resolv kernel ep} for $K$ gives the conclusion.\\
\end{proof}

\begin{lemma}\label{l:rbigmbig}
 There are $C$ and $h_0$ such that \eqref{bound cutoff kernel} holds for all $r$ and $r'$ in $[R_0,\infty)$ obeying $r \le r'$, $h \in(0,h_0]$, $E\in[E_{\min},E_{\max}]$, and for all $m > M_+$.
\end{lemma}

\begin{proof}
By the Bessel function differential equation \eqref{e:wdifeq}, there are real numbers $A$ and $B$  such that, for $r \ge R_0$,
\[
u_0(r) = r^{1/2}(A J_\nu(\nu z)+ B Y_\nu(\nu z)),
\]
where
\[
\nu = h^{-1}(m + \tfrac 14 h^2)^{1/2}, \qquad z = (m + \tfrac 14 h^2)^{-1/2} E^{1/2} r. 
\]
(Note that the constants $A$ and $B$ here are analogous to but different from the ones from Lemma \ref{Olver lemma}.) Recall from \eqref{defn u1} that  
\[
u_1(r) = r^{1/2}(J_\nu(\nu z) + i Y_\nu(\nu z)).
\]

By the Bessel function Wronskian formula \eqref{e:wwronskian},
\[
W = 2\pi^{-1}(iA-B).
\]
To bound $B$ in terms of $A$ we use the fact that $u_0(R_0), \, u'_0(R_0) \ge 0$ which follows from \\ $V + mr^{-2} - E \geq 0$ on $(0,R_0]$; see Lemma~\ref{nonnegativity}. By the Bessel function bounds \eqref{e:jyzsmall}, for $h$ small enough, we have
\[
B \le \frac {J_\nu(\nu z_0)}{-Y_\nu(\nu z_0)}A \lesssim e^{-2 \nu \xi_0} A, \qquad -B \le \frac{J_\nu(\nu z_0)+2R_0J'_\nu(\nu z_0)}{Y_\nu(\nu z_0)+2R_0Y'_\nu(\nu z_0)} A \lesssim e^{-2 \nu \xi_0}  A
\]
where
\[
 \xi_0 = \int_{z_0}^1 t^{-1} (1-t^2)^{1/2}dt, \qquad z_0 = (m + \tfrac 14 h^2)^{-1/2} E^{1/2} R_0. 
\]
Note that $z_0 <1$ because 
\begin{equation*}
z_0 \le M_+^{-1/2} E^{1/2} R_0 = \left(\frac{M_0}{ER^2_0} +\frac{R^2 - R^2_0}{2R^2_0}\right)^{-1/2} \le \left(1 +\frac{R^2 - R^2_0}{2R^2_0}\right)^{-1/2}.
\end{equation*}

Then, when  $R_0 \le r \le r'$, letting $z' = (m + \tfrac 14 h^2)^{-1/2} E^{1/2} r'$ and inserting $|B| \lesssim e^{-2\nu \xi_0} A$ and the Bessel function bound \eqref{e:jyzbigbound} into \eqref{resolv kernel ep} gives
\begin{equation}\label{e:krr'big}\begin{split}
 |K(r,r') | &\lesssim h^{-2} (rr')^{1/2}(J_\nu(\nu z)+e^{-2 \nu \xi_0} |Y_\nu(\nu z)|) |J_\nu(\nu z')+iY_\nu(\nu z')| \\&\lesssim h^{-2}\nu^{-1} (rr')^{1/2}\langle z\rangle^{-1/2}\langle z'\rangle^{-1/2} \lesssim h^{-1},
\end{split}\end{equation}
as desired. \\
\end{proof}

\section{One dimensional resolvent estimates for fixed angular momentum}\label{s:low}

We now study the equation
$$
(P_{m}-E-i\ep)  u =(-h^2\partial_r^2+V(r)+mr^{-2}-E - i\ep)u=f,\qquad \quad m \ge - \frac{h^2}{4},
$$
on $L^2 = L^2(\R_+) \defeq L^2(0, \infty)$. Recall that the self-adjoint extension of $(P_m, C^\infty_0(\R_+))$ we employ is its Friedrichs extension $(P_m, \dom_m)$. Put $\dom_{\text{comp},m} \defeq \{u \in \dom_m : \supp u \text{ is compact in $[0,\infty)$}\}$. 
\begin{lemma}
\label{l:oneDResolve}
Let $[E_{\min},E_{\max}] \subseteq (0, \infty)$ and $1/2 < s \le 1$. Then there is $C, \, h_0 >0$ such that for all $E\in [E_{\min},E_{\max}]$, $h \in (0, h_0]$, $m \ge -h^2/4 $, $0 \le  \ep \leq 1$, and $u\in \dom_{\emph{comp},m}$,
$$
\|\langle r\rangle^{-s}u\|_{L^2(\R_+)}\leq e^{C(1+|m|^{1/2})/h}\|\langle r\rangle^{s}(P_m - E - i\ep)u\|_{L^2(\R_+)}.
$$
\end{lemma}

To prove Lemma~\ref{l:oneDResolve}, we start by studying the operator
$$
Q_m \defeq -\partial_r^2+mh^{-2}r^{-2}, \qquad m \ge - \frac{h^2}{4},
$$
on functions compactly supported in $[0,\infty)$. We recall the Mellin transform and its inverse:
$$
\mathcal{M}(u)(\sigma):=\int_0^\infty r^{i\sigma} u(r)\frac{dr}{r},\qquad \mathcal{M}_t^{-1}(v)(r):=\frac{1}{2\pi}\int_\R r^{-i\sigma}v(\sigma)d\tau, \qquad \sigma = \tau + it,
$$
where the definitions hold initially, e.g., for $u \in C_0^\infty(\R_+)$ and $v( (\cdot) + it) \in L^1_\tau(\R) \cap L^2_\tau(\R)$, and then extend by density to bounded operators $\mathcal{M} : L^2(0, \infty ; r^{-2t -1}dr) \to L^2_\tau(\R)$, \\$\mathcal{M}_t^{-1} : L^2_\tau(\R) \to L^2(0, \infty ; r^{-2t -1}dr)$. Moreover, since $\mathcal{M}(u)(\tau + it) = 2 \pi \mathcal{F}^{-1}( e^{-tx} u(e^{x}))(\tau)$, $x \in \R$, and $\mathcal{M}^{-1}_t(v)(r) = r^t \mathcal{F}(v)(\log r)/2\pi$, where $\mathcal{F}$ denotes Fourier transform,
\begin{equation} \label{almost unitarity}
\begin{gathered}
\|\mathcal{M}(u)(\tau+it)\|_{L^2_\tau(\mathbb{R})} = (2\pi)^{1/2}\|r^{-t-1/2}u\|_{L^2(\R_+)}, \\
\| r^{-t-1/2}\mathcal{M}_t^{-1}(v)(r)\|_{L^2(\R_+)}= (2\pi)^{-1/2}\|v\|_{L^2_\tau(\mathbb{R})}. 
\end{gathered}
\end{equation}

Let 
$$
t_\pm = t_{\pm}(m):=\frac{1\pm\sqrt{1+4mh^{-2}}}{2},\qquad \Lambda(t,m):=| t^2-t-h^{-2}m|^{-1}, \qquad t \neq t_\pm.
$$ 

\begin{lemma}
\label{l:nearZero}
There is $C>0$ such that for $m\geq -h^2/4$, $N\in \mathbb{R}$, $t_0 \in \mathbb{R}\setminus\{ t_+(m),t_-(m)\}$, and $u\in r^{-N}L^2_{\comp}[0, \infty)$ with $Q_mu\in r^{t_0-\frac{3}{2}}L^2$, we have
\begin{equation} \label{id from Mellin transf}
u= \Pi_{t_0}(r^2Q_mu)+E_{t_0}(r^2Q_mu),
\end{equation}
where
\begin{equation} \label{E sub t est} 
\|r^{-t_0-\frac{1}{2}}E_{t_0}v\|_{L^2}\leq C \Lambda(t_0,m) \|r^{-{t_0}-\frac{1}{2}}v\|_{L^2(\R_+)},
\end{equation}
and 
$$
\Pi_{t_0}v=\begin{cases} 0&t_0<t_-,\\
 r^{\frac{1}{2}}\log r\mathcal{M}(v)(\tfrac{i}{2})-r^{\frac{1}{2}}\mathcal{M}(\log rv)(\tfrac{i}{2})&t_-<t_0,\,m=-\tfrac{h^2}{4},\\
\frac{r^{t_-}\mathcal{M}(v)(it_-)}{t_--t_+}&t_-<t_0<t_+,\,m>- \tfrac{h^2}{4},\\
\frac{r^{t_-}\mathcal{M}(v)(it_-)}{t_--t_+}+\frac{r^{t_+}\mathcal{M}(v)(it_+)}{t_+-t_-}&t_+<t_0,\, m>-\tfrac{h^2}{4}.
\end{cases}
$$
\end{lemma}
\textbf{Remark:}  Applying $(r \partial_r)^j$, $j = 1, \, 2$, to the expression for $E_{t_0} v$, see \eqref{decompose E Pi} below, yields a strengthening of \eqref{E sub t est}, namely 
$$\|r^{-t_0-\frac{1}{2}}(r\partial_r)^jE_{t_0}v\|_{L^2}\leq \Lambda_j(t_0,m) \|r^{-{t_0}-\frac{1}{2}}v\|_{L^2(\R_+)},\qquad j=0, \, 1, \, 2,
$$
However, we omit the proof since in the sequel we do not need the estimates for $j = 1, \, 2$.
\begin{proof}
Without loss of generality, we take $N$ positive and large enough so that $-N < t_0,\, t_+, \, t_-$. Since $u\in r^{-N}L_{\comp}^2$ and $r^2Q_mu\in r^{t_0+\frac{1}{2}}L^2_{\comp}$, $\mathcal{M}(u)(\sigma)$ is holomorphic in $\Im \sigma<-N- 1/2$, while $\mathcal{M}(r^2Q_mu)(\sigma)$, is holomorphic in $\Im \sigma<t_0$ and extends continuously (as an $L^2$ function) to $\Im \sigma=t_0$. In addition, 
\begin{equation} \label{relate Mellin transf}
\mathcal{M}(r^2Q_mu)(\sigma)=(\sigma^2 - i \sigma + mh^{-2})\mathcal{M}(u)(\sigma),\qquad \Im \sigma<-N - \frac{1}{2}.
\end{equation}
In particular, \eqref{relate Mellin transf} implies $\mathcal{M}{u}(\sigma) \in L^1_\tau(\R) \cap L^2_\tau(\R)$ for $\Im \sigma < - N - 1/2$, so by Fourier inversion,
\begin{equation} \label{Mellin inv u}
u(r)=\frac{1}{2\pi} \int_{\Im \sigma=-N-1} \frac{r^{-i\sigma}\mathcal{M}(r^2Q_mu)(\sigma)}{\sigma^2-i\sigma+mh^{-2}}d\sigma.
\end{equation}
We now deform the contour to $\Im\sigma=t_0-\e$ and the send $\e\to 0$. From \eqref{Mellin inv u},
$$
u(r)=\frac{1}{2\pi} \lim_{R\to \infty}\int_{\gamma_{R,-N-1}} \frac{r^{-i\sigma}\mathcal{M}(r^2Q_mu)(\sigma)}{\sigma^2-i\sigma+mh^{-2}}d\sigma,\qquad 
\gamma_{R,t}=\{ \tau+it\,:\, \tau \in[-R,R]\}.
$$
Next, observe that since, in $\Im \sigma<t_0$, $\|\mathcal{M}(r^2Q_mu)(\sigma)\|_{L_\tau^\infty(\R)}$ is independent of $\Re \sigma$,
$$
\Big|\int _{\gamma_{\pm,R,-N, \e}} \frac{r^{-i\sigma}\mathcal{M}(r^2Q_mu)(\sigma)}{\sigma^2-i\sigma+mh^{-2}}d\sigma\Big|\leq \frac{C_\e}{R^2} \to 0, \qquad \text{as $R \to \infty$.} 
$$
where 
$$\gamma_{\pm,R,-N, \e}:=\{ \pm R+it\,:\, t\in [-N-1,t_0-\e]\}.$$
In particular, using $t_0\neq t_{\pm}(m)$ and that $\mc{M}(r^2Q_mu)(\tau + it) $ varies continuously in $L^2_\tau(\R)$ for $t \le t_0$, we send $\e\to 0$ to obtain

\begin{equation} \label{decompose E Pi}
\begin{split}
u(r)&=\frac{1}{2\pi} \int_{\Im \sigma=t_0}\frac{r^{-i\sigma}\mathcal{M}(r^2Q_mu)(\sigma)}{\sigma^2-i\sigma+mh^{-2}}d\sigma + i \sum_{\substack{\sigma^2-i\sigma+mh^{-2}=0\\\Im \sigma<t_0}}\operatorname{Res}\Big(r^{-i\sigma}\frac{\mathcal{M}(r^2Q_mu)(\sigma)}{\sigma^2-i\sigma +mh^{-2}}\Big) \\
&=:E_{t_0}(r^2Q_mu)+\Pi_{t_0}(r^2Q_mu).
\end{split}
\end{equation}

The formula for $\Pi_{t_0}$ follows from calculating residues at $t_{\pm}(m)$, while the bound on $E_{t_0}$ follows from minimizing the modulus
$$
|(\tau + it_0)^2-i(\tau + it_0) +mh^{-2}|^2 = (-t_0^2 + t_0 +mh^{-2} + \tau^2)^2 + \tau^2(2t_0-1)^2,
$$
with respect to $\tau$.\\
\end{proof}

\begin{lemma}
\label{l:domain}
Suppose $m> 0$ and $u\in \dom_{m}$. Then $r^{-1}u\in L^2(\R_+)$. 
\end{lemma}
\begin{proof}
Observe that there exists $C> 0$ so that, 
$$
m \|r^{-1} u \|_{L^2(\R_+)}^2 \le C\|u\|_{L^2(\R_+)}^2 +  |\langle P_m u, u \rangle_{L^2(\R_+)}|, \qquad u\in C^\infty_0(\R_+).
$$
As $(P_m, \dom_m)$ is the Friedrichs extension of $(P_m, C_0^\infty(0, \infty))$, by \cite[Theorem X.23]{reedsimon2}, any $u \in \dom_m$ has a sequence $\{  \psi_k \} \subseteq C^\infty_0(\R_+)$ such that $\psi_k \to u$ in $L^2(\R_+)$ and $\langle P_m \psi_k, \psi_k \rangle_{L^2}\to \langle P_m u, u \rangle_{L^2(\R_+)} $. We also get pointwise almost everywhere convergence by extracting a subsequence of the $\psi_k$, which we again just call $\psi_k$. Sending $k \to \infty$ and using Fatou's lemma gives \\ $m \|r^{-1}u\|_{L^2(\R_+)}^2 \le C\|u\|_{L^2(\R_+)}^2 +  \|P_m u \| _{L^2(\R_+)}\| u\|_{L^2(\R_+)} < \infty $ for general $u \in \dom_m$ as desired. \\
\end{proof} 

Lemma~\ref{l:nearZero} will allow us to control the behavior of solutions, $u$ to $(P_m - E - i\ep)u=f$ near $0$. 
\begin{proof}[Proof of Lemma~\ref{l:oneDResolve}]
To shorten notation, we set $P_{m,E,\ep}=P_m-E-i\ep$. Since $u\in \dom_{\comp,m}$ and $u\in L^2$, for any compact $K\subseteq (0,\infty)$, there exists $C_{K,h} > 0$ so that  
$$
\|\partial_r^2u\|_{L^2(K)}\leq C_{K,h}(\|u\|_{L^2(K)}+\|P_{m}u\|_{L^2(K)}).
$$ 
In particular, $u\in H^2_{\loc}(0,\infty)$.

Next, set 
\begin{equation} \label{defn deltas}
\delta = \delta(m):=\delta_0 
h\langle h^{-2}m\rangle^{1/2},\qquad \delta_1 = \delta_1(m) \defeq \max(\delta(m),\tfrac{1}{2}),
\end{equation}
for some $0<\delta_0\ll 1$ independent of $h$, $m$ and $E\in [E_{\min},E_{\max}]$, to be chosen. Let $\chi\in C_0^\infty[0,2)$ with $\chi \equiv 1$ near $[0,1]$. Set $\chi_{\delta_1}=\chi(\delta^{-1}_1 r)$. 
Then
\begin{equation} \label{rewrite Q m}
\begin{split}
Q_m\chi_{\delta_1} u &= h^{-2}P_{m,E,\ep}\chi_{\delta_1}  u -h^{-2}(V-E - i\ep)\chi_{\delta_1} u\\
&=h^{-2}\chi_{\delta_1}  P_{m}u +h^{-2}[P_{m},\chi_{\delta_1}  ]u-h^{-2} \chi_{\delta_1} (V-E - i\ep)u.
\end{split}
\end{equation}
In particular, since $u\in  H^2_{\loc}(0,\infty)$, $P_mu\in L^2$, and $\chi_{\delta_1}$ is constant near zero, $r^2Q_m \chi_{\delta_1} u\in r^{2}L^2$.  Setting 
$$
t_0 = t_0(m) \defeq \begin{cases}-\frac{1}{2}& -\frac{h^2}{4}\leq m\leq \frac{h^2}{4},\\\ 1&\frac{h^2}{4}<m,
\end{cases}
$$
observe that
\begin{equation} \label{est Lambda}
\Lambda(t_0,m)=|t^2_0 - t_0- h^{-2}m|^{-1}\leq c\langle h^{-2}m\rangle^{-1}
\end{equation}
for some $c > 0$ independent of $h$ and $m$.
Note also that with this definition of $t_0(m)$, $t_0(m)<t_-(m)$ for $-h^2/4 \leq  m \leq h^2/4$, and $t_-(m)<t_0(m)<t_+(m)$ for $h^2/4<m$. Therefore, by Lemma~\ref{l:nearZero}, since $u\in L^2$, $r^2Q_m\chi_{\delta_1} u\in r^2L^2$, and $t_0(m)\leq 3/2$,
$$
\chi_{\delta_1} u= E_{t_0}(r^2Q_m\chi_{\delta_1} u)+\Pi_{t_0}(r^2Q_m\chi_{\delta_1} u),
$$
where
$$
\Pi_{t_0}(r^2Q_m\chi_{\delta_1} u)=\begin{cases}0&-\frac{h^2}{4}\leq m\leq \frac{h^2}{4},\\\frac{r^{t_-}}{t_- - t_+ }\mathcal{M}(r^2Q_m\chi_{\delta_1} u)(it_-)&\frac{h^2}{4}<m,\end{cases}
$$
and
$$\|r^{-t_0-\frac{1}{2}} E_{t_0}v\|_{L^2}\leq C \Lambda(t_0,m) \|r^{-t_0-\frac{1}{2}}v\|_{L^2}.$$

Next, by Lemma~\ref{l:domain}, for $m> 0$, $r^{-1}u\in L^2$. Thus, both $\mc{M}(\chi_{\delta_1 }u)(\sigma)$ and $\mathcal{M}(r^2Q_m\chi_{\delta_1} u)(\sigma)$ are holomorphic in $\imag \sigma < 1/2$. As $t_- < 1/2$ when $m > 0$, by \eqref{relate Mellin transf},
$$
\mathcal{M}(r^2Q_m\chi_{\delta_1} u)(it_-)= ((it_-)^2-i(it_-)-h^{-2}m)\mc{M}(\chi_{\delta_1 }u)(it_-)=0.
$$
In particular, for any $m$ and $u\in \mc{D}_{\comp,m}$, $\Pi_{t_0}(r^2Q_m\chi_{\delta_1} u)=0$.
Thus, by \eqref{rewrite Q m} and $\delta_1 \ge 1/2$,
\begin{equation}
\label{e:squigglyA}
\begin{aligned}\|r^{-t_0-\frac{1}{2}}u\|_{L^2(0,\delta_1)}&\leq \|r^{-t_0-\frac{1}{2}}\chi_{\delta_1} u\|_{L^2(0,\delta_1)} \leq C\Lambda(t_0,m)\|r^{\frac{3}{2}-t_0}Q_m\chi_{\delta_1} u\|_{L^2}\\
&\leq C\Lambda(t_0,m)\Big(h^{-2}\|r^{\frac{3}{2}-t_0}P_{m,E,\ep}u\|_{L^2(0,2\delta_1)}\\
&+ h^{-1} \|r^{\frac{3}{2}-t_0}hu'\|_{L^2(\delta_1,2\delta_1)}+h^{-2}\|r^{\frac{3}{2}-t_0}u\|_{L^2(0,2\delta_1)}\Big).\end{aligned}\end{equation}
We now estimate part of the last term on the right side of \eqref{e:squigglyA}, using \eqref{defn deltas} and \eqref{est Lambda}:
\begin{equation*}
\begin{split}
C\Lambda(t_0,m) h^{-2}\|r^{\frac{3}{2}-t_0}u\|_{L^2(0,\delta)} & \le C\Lambda(t_0,m) h^{-2} \delta^2 \|r^{-t_0 - \frac{1}{2}}u\|_{L^2(0,\delta)}  \\
&\le C c \delta_0 \|r^{-t_0 - \frac{1}{2}}u\|_{L^2(0,\delta)}.
\end{split}
\end{equation*}
Choosing $\delta_0$ small enough, this term may be absorbed into the left side of \eqref{e:squigglyA}, so we find
\begin{equation}
\label{e:squiggly}
\begin{aligned}\|r^{-t_0-\frac{1}{2}}u\|_{L^2(0,\delta_1)}&\leq Ch^{-2}\langle h^{-2}m\rangle^{-1}\Big(\|r^{\frac{3}{2}-t_0}P_{m,E,\ep}u\|_{L^2(0,2\delta_1)}+\\
& h \|r^{\frac{3}{2}-t_0}hu'\|_{L^2(\delta_1,2\delta_1)}+\|r^{\frac{3}{2}-t_0}u\|_{L^2(\delta ,2\delta_1)}\Big).
\end{aligned}
\end{equation}

We now employ the energy method to study the region $[\delta(m),\infty)$. Let $s> 1/2$, $\eta = \min\{1, E\}/2$, and
 $\phi_j\in C_0^\infty([0,1);[0,1])$, $j=0,1,2$, with $\phi_0 \equiv 1$ on $[0,1/2]$ and $\phi_j\equiv 1$ on $\supp \phi_{j-1}$. Let
$$
F(r):= |hu'(r)|^2+E|u(r)|^2,\qquad w(r):=\int_0^r(1-\phi_1( r' /\delta))\phi_2(r' /\delta )d r' e^{\psi(r)/h},
$$
with
\begin{equation*}
\psi \defeq \int_0^r \eta^{-1} \big( |V( r')| + 2|m|(1-\phi_0( r'/\delta))(r')^{-2} \big) +\langle r' \rangle^{-2s}d r' \\
\leq \eta^{-1}\|V\|_{L^1(\R_+)}+C(1+ \eta^{-1} |m|^{1/2}).
\end{equation*}
Here, we have used
\begin{equation*}
\begin{split}
|m| \int^r_{\tfrac{1}{2} \delta(m)} (r')^{-2} dr' = |m| \Big( (\tfrac{1}{2} \delta(m))^{-1} - r^{-1} \Big)  \lesssim |m|^{1/2}. 
\end{split}
\end{equation*}
Then,
\begin{align*}
(wF)'(r)&= -2\Re w\langle P_{m,E,\ep}u,u'\rangle + 2\ep w\Im \langle u,u'\rangle +w'(|hu'|^2 +E|u|^2)+2w\Re \langle (V+mr^{-2})u,u'\rangle,
\end{align*}
where $\langle z, z_1 \rangle \defeq z \overline{z_1}$, $z, z_1 \in \C$. Since $\phi_0(r/\delta) > 0$ implies $w(r) = 0$, we have
\begin{equation}
\label{e:derivativeBig}
\begin{aligned}
 w'&=\frac{\psi'}{h}w+ (1-\phi_1(r/\delta))\phi_2(r/\delta)) e^{\psi/h} \geq \frac{\psi'}{h}w\geq h^{-1}w[\eta^{-1}|V|+\langle r\rangle^{-2s}+ \eta^{-1}|m|r^{-2}].
 \end{aligned}
\end{equation}
So,
\begin{align*}
(wF)'(r)&\geq -2\Re w\langle P_{m,E,\ep}u,u'\rangle + 2\ep w\Im \langle u,u'\rangle +w'(|hu'|^2 +E|u|^2)+2w|\Re \langle (V+mr^{-2})u,u'\rangle|\\
&\geq  - h^{-1}\eta^{-1}\langle r\rangle^{2s}w|P_{m,E,\ep}u|^2 -h^{-1}\eta w\langle r\rangle^{-2s}|hu'|^2\\
&-h^{-1}(|V|+\ep+|m|r^{-2})w(|u|^2+|hu'|^2)+ \min\{1, E\} w'(|hu'|^2 +|u|^2)\\
&\geq - h^{-1}\eta^{-1}\langle r\rangle^{2s}w|P_{m,E,\ep}u|^2 - \ep h^{-1} w(|u|^2+|hu'|^2)+ \frac{\min\{1, E\}}{2}w'(|hu'|^2 +E|u|^2).
\end{align*}

We integrate from $0$ to $\infty$, and use the facts that $w(r) = 0$ near zero and $u\in H^{2}_{\loc}((0,\infty))$ is compactly supported in $[0,\infty)$:
\begin{equation}
\label{e:wEst}
\begin{aligned}
\|(w')^{1/2}(hu')\|^2_{L^2(0,\infty)}&+\|(w')^{1/2}u\|^2_{L^2(0,\infty)}\\
&\leq  Ch^{-1} \int \langle r\rangle^{2s}w(r)|P_{m,E,\ep}u(r)|^2dr+\ep h^{-1}\int w(|u(r)|^2+|hu'(r)|^2)dr.
\end{aligned}
\end{equation}
Next, observe that 
$$\Re \int w P_{m,E,\ep}u\bar{u}dr= \int w|hu'|^2dr+ \Re \int hw'hu'\bar{u}dr +\int w(V-E+mr^{-2})|u|^2dr.$$
So, using~\eqref{e:derivativeBig}
\begin{equation*}
\begin{split}
\|(w)^{1/2}hu'\|_{L^2}&\leq \frac{ \gamma h}{2}\| (w')^{1/2}hu'\|^2_{L^2}+ \frac{h\gamma^{-1} }{2}\| (w')^{1/2}u\|^2_{L^2}+C\|(w)^{1/2}u\|_{L^2}^2\\&+C\|w^{1/2}|m|^{1/2}r^{-1}u\|_{L^2}^2+\frac{1}{2}\|w^{1/2}\langle r\rangle^{s}P_{m,E,\ep}u\|_{L^2}+\frac{1}{2}\|w^{1/2}\langle r\rangle^{-s}u\|_{L^2}\\
&\leq \frac{\gamma h}{2}\| (w')^{1/2}hu'\|^2_{L^2}+ C(1+\gamma^{-1})h\| (w')^{1/2}u\|^2_{L^2}\\
&+\frac{1}{2}\|w^{1/2}\langle r\rangle^{s}P_{m,E,\ep}u\|_{L^2}+C\|(w)^{1/2}u\|_{L^2}.
\end{split}
\end{equation*}
Plugging this into the right side of \eqref{e:wEst}, taking $\gamma$ small enough, subtracting the corresponding term to the left-hand side, and using,
\begin{equation*}
\label{e:wInfo}
 \begin{gathered}c\delta(m)\langle r\rangle^{-2s}/h\leq w'(r),\qquad r\geq \delta(m), \\
  |w(r)|, |w'(r)| \leq e^{C(1+|m|^{1/2})/h}, \qquad   r \geq 0,\\
 \supp w\subseteq (\tfrac{1}{2}\delta(m),\infty),
 \end{gathered}
\end{equation*}
and $\ep \leq 1$, we have
\begin{equation}
\label{e:squigglyB}
\begin{aligned}
& \|\langle r\rangle^{-s}(hu')\|_{L^2(\delta ,\infty)}+\|\langle r\rangle ^{-s}u\|_{L^2(\delta ,\infty)}\\
&\leq  e^{C(1+|m|^{1/2})/h}\|\langle r\rangle^{s}P_{m,E,\ep}u\|_{L^2(0,\infty)}+\ep^{1/2}e^{C(1+|m|^{1/2})/h}\|u\|_{L^2(\frac{1}{2}\delta ,\infty)}.
\end{aligned}
\end{equation}
Next, observe that $-1/2\leq t_0(m)\leq 3/2$ implies
\begin{gather*}
\|r^{\frac{3}{2}-t_0}v\|_{L^2(0,2\delta_1)}\leq C \delta_1^{t_0-\frac{3}{2}}\|v\|_{L^2(0,2\delta_1)},\\ \|v\|_{L^2(0,\delta)}\leq C\delta^{t_0+\frac{1}{2}}\|r^{-t_0-\frac{1}{2}}v\|_{L^2(0, \delta)},\\
\|r^{\frac{3}{2}-t_0}v\|_{L^2(\delta_1,2\delta_1)}\leq C\delta_1^{t_0-\frac{3}{2}+s}\|\langle r\rangle^{-s}v\|_{L^2(\delta_1,2\delta_1)}.
\end{gather*}
Combining this with~\eqref{e:squiggly} and~\eqref{e:squigglyB}, and that $\delta_1\leq C\langle m\rangle^{1/2}$, we have 
$$
\begin{aligned}
\|\langle r\rangle^{-s}u\|_{L^2(0,\infty)}\leq e^{C(1+|m|^{1/2})/h}&\|\langle r\rangle^{s}P_{m,E,\ep}u\|_{L^2(0,\infty)}+\ep^{1/2} e^{C(1+|m|^{1/2})/h}\|u\|_{L^2(\frac{1}{2}\delta ,\infty)}.
\end{aligned}
$$
Finally, observe that, since $u \in \dom_m$, for any $\gamma > 0$,
$$
\ep \|u\|^2_{L^2}= |\Im \langle P_{m,E,\ep} u,u\rangle|\leq \|\langle r\rangle^{s}P_{m,E,\ep}u\|_{L^2}\| \langle r\rangle^{-s} u\|_{L^2} \le \tfrac{1}{2\gamma}\|\langle r\rangle^{s}P_{m,E,\ep}u\|^2_{L^2} + \tfrac{\gamma}{2} \| \langle r\rangle^{-s} u\|^2_{L^2},$$
and hence
$$
\begin{aligned}
\|\langle r\rangle^{-s}u\|_{L^2(0,\infty)}\leq e^{C(1+|m|^{1/2})/h}& \|\langle r\rangle^{s}P_{m,E,\ep}u\|_{L^2(0,\infty)}.
\end{aligned}
$$
This completes the proof of Lemma~\ref{l:oneDResolve}.\\
\end{proof}

\begin{proof}[Proof of Proposition~\ref{p:1DResolve}]
Let $1/2<s< 1$, and $W_s\in C^\infty$ with $W_s\equiv 1 $ near $0$ and \\ $c\langle r\rangle^s\leq W_s(r) \leq C\langle r\rangle^{s}$, and $W_s=C\langle r\rangle^s$ on $r \geq 2$. 

First, recall that for any $\chi \in C_0^\infty[0, \infty)$ with $\chi \equiv 1$ near $0$, 
$$
\|u\|_{\dom_m}\sim \|\chi u\|_{\dom_m}+\|(1-\chi)u\|_{H^2}.
$$
Thus, since $[W_s,P_m](W_s)^{-1}:H^2\to H^1$ is supported away from zero, for any $v \in \dom_m$ such that $W_s v \in \dom_m$,
\begin{equation}
\label{e:wsBound}
\|W_s(P_m-E-i\ep)v\|_{L^2}\leq \|[W_s,P_m](W_s)^{-1} W_sv\|_{L^2}+\|(P_m-E-i\ep)W_sv\|_{L^2}\leq C_{\ep, h} \|W_sv\|_{\dom_m}.
\end{equation}

Next, we claim that for $\ep>0$ and $f\in L^2(\R_+)$,
$$
 W_s(P_m-E-i\ep)^{-1}\langle r\rangle^{-s}f\in \dom_m.
$$
Put 
 $$
v \defeq (P_m-E-i\ep)^{-1}\langle r\rangle^{-s}f\in \dom_m.
 $$
 We shall show $W_s v \in \dom_m$ by demonstrating that 
 \begin{equation} \label{Ws v}
 W_s v = (P_m-E-i\ep)^{-1}\langle r\rangle^{-s}(W_s \langle r \rangle^{-s} f + g)\
 \end{equation}
 for suitable $g \in L^2(\R_+)$. 
 
Let $\chi \in C_0^\infty([0,\infty); [0,1])$ with $\chi \equiv 1$, and put $\chi_{k} = \chi(k^{-1} r)$, $ k \in \NN$.
 Almost everywhere on $\R_+$,
 $$
 (P_m - E - i\ep) \chi_k W_s v = \chi_k W_s \langle r\rangle^{-s} f + [P_m, \chi_k W_s]v.
$$
So in particular $P_m \chi_k W_s v \in L^2(\R_+)$, hence $\chi_k W_s v \in \dom(P^*_m)$. This implies that, for each $k$, 
\begin{equation} \label{Ws v prelim}
\chi_k W_s v = (P_m - E - i\ep)^{-1}(\chi_k W_s \langle r\rangle^{-s} f + [P_m, \chi_k W_s]v).
\end{equation}
We complete the proof of \eqref{Ws v} by sending  $k \to \infty$ in \eqref{Ws v prelim} and noting that the $[P_m, \chi_k W_s]$, $[P_m,  W_s]$ are zero on a fixed neighborhood of $r = 0$ (independent of $k$) and that they are uniformly bounded $H^2(\R_+) \to L^2(\R_+)$ because $s \le 1$.

Now, let $v_k=\chi(k^{-1}r)W_sv\in \dom_{\comp, m}$, $k \in \NN$, and note that $v_k \to  W_sv$ in $\dom_m$. Observe that 
\begin{equation}
\label{e:L2Approx}
\|\langle r\rangle^{-s}((W_s)^{-1} v_k-v)\|_{L^2}\leq \|(v_k-W_sv)\|_{L^2}\to 0,
\end{equation}
and by~\eqref{e:wsBound},
\begin{equation}
\label{e:PApprox}
\begin{aligned}
\|\langle r\rangle^s(P-E-i\ep)((W_s)^{-1}v_k-v)\|_{L^2}&\leq C\|W_s(P_m-E-i\ep)((W_{s})^{-1}v_k-v)\|_{L^2}\\
&\leq C_{\ep, h} \|v_k-W_sv\|_{\dom_m}\to 0.
\end{aligned}
\end{equation}
Finally, applying Lemma~\ref{l:oneDResolve},
$$
\|\langle r\rangle^{-s} (W_s)^{-1}v_k\|_{L^2} \leq Ce^{C(1+|m|^{1/2})/h}\|\langle r\rangle^{s}(P_m - E - i\ep)(W_s)^{-1}v_k\|_{L^2}.$$
Sending $k\to \infty$ and using~\eqref{e:L2Approx},~\eqref{e:PApprox}, and the definition of $v$, this implies
$$
\|\langle r\rangle^{-s}(P_m-E-i\ep)^{-1}\langle r\rangle^{-s}f\|_{L^2}\leq Ce^{C(1+|m|^{1/2})/h}\|f\|_{L^2},
$$
completing the proof of the proposition.\\
\end{proof}

\section{Exponential estimates for the resolvent}\label{s:exp thm}
In this section we prove Theorem \ref{exp thm}. As before, we use separation of variables to reduce estimating the resolvent of $P$ to estimating the resolvent of each $P_m$.
By Proposition~\ref{p:1DResolve}, for every $m\geq -h^2/4$ and $0\leq \ep \leq 1$
\[
 \|\langle r \rangle^{-s} (P_m-E-i\ep)^{-1}\langle r \rangle^{-s} \|_{L^2(\R_+) \to L^2(\R_+) } \le e^{C(|m|^{1/2}+1)/h}.
\]
Thus it is enough to prove the following lemma.
\begin{lemma}
There are $C$ and $h_0$ such that
\[
|K(r, r')| = |K(r,r',m,E,0,h)| \le \frac C{h^2}
\]
holds whenever $0 < r \le r'$, $m \ge M_+$, and $0<h<h_0$.
\end{lemma}
\begin{proof}
We use the fact that $u_0(0) = 0$ and $u_0'(r) \ge 0$ when $r \le R_0$, and consider separately three cases.

\begin{enumerate}
 \item Suppose that $R_0 \le r \le r'$. Then the result follows from the stronger estimate proved in Lemma~\ref{l:rbigmbig}.

\item Suppose that $r \le R_0 \le r'$. Then
\[
 |K(r,r')|  =  u_0(r)\left| \frac{ u_1(r')}{h^2W} \right| \le u(R_0) \frac{ |u_1(r')|}{h^2 |W|} = |K(R_0, r')|  \lesssim h^{-1},
\]
where we used $u_0' \ge 0$ on $(0,R_0)$ and then \eqref{e:krr'big}.

\item Suppose that $r \le r' \le R_0$. Dividing the definition of the Wronskian, $u_0u_1' - u_0' u_1 = W$,  by $u_0^2$ and integrating both sides gives
\[
 u_1(r) = u_0(r) \Big(\frac {u_1(R_0)}{u_0(R_0)} + \int_{R_0}^r \frac W {u_0(s)^2}ds \Big),
\]
thus, using again $u_0' \ge 0$, on $(0,R_0)$
\[
 |u_0(r)u_1(r')| \le u_0(R_0)|u_1(R_0)| + (R_0-r')|W|.
\]
Plugging into \eqref{resolv kernel ep}, using \eqref{e:krr'big} on the first term, and estimating the second term directly gives the result.
\end{enumerate}
\end{proof}

\appendix

\section{Bessel functions}

In this paper we use the Bessel functions $J_\nu(z)$ and $Y_\nu(z)$ \cite[Section 10.2]{dlmf}, where $\nu\geq 0$ and $z>0$, and below we review some standard facts about them.  By \cite[Equation 10.13.1]{dlmf}, the differential equation
\begin{equation}\label{e:wdifeq}
\partial_r^2w + \left( \lambda^2 - r^{-2}(\nu^2 - \tfrac 14)\right)w = 0,
\end{equation}
is solved by $w_1= r^{1/2}J_\nu(\lambda r)$ and $w_2=r^{1/2}Y_\nu(\lambda r)$, for any $\lambda>0$. By \cite[Equation 10.5.2]{dlmf}, we have the Wronskian formula
\begin{equation}\label{e:wwronskian}
w_1 \partial_r w_2 - w_2 \partial_rw_1 = 2 \pi^{-1}.
\end{equation}

We use upper and lower bounds for $J$ and $Y$ derived from Olver's uniform asymptotics for large values of $\nu$. To state them, we use the notation $a \asymp b$ to mean $a \lesssim b$ and $b \lesssim a$. We define a decreasing bijection  $(0,\infty) \ni z \mapsto \zeta(z) \in \mathbb R$ by
\begin{equation}\label{e:zetadef}
 \zeta = \begin{cases}
 \phantom{-}\frac 32 \big(\int_z^1 t^{-1} (1-t^2)^{1/2}dt\big)^{2/3}, \qquad &z \le 1, \\ 
 -\frac 32 \big(\int_1^z t^{-1} (t^2-1)^{1/2}dt\big)^{2/3}, \qquad &z \ge 1,
\end{cases}
 \end{equation}
and use the Airy functions 
\begin{equation}\label{e:airydef}\begin{split}
 \Ai(x) &= \frac 1 \pi \int_0^\infty \cos\left(\tfrac {t^3}3 + xt\right)dt, \\
 \Bi(x) &= \frac 1 \pi \int_0^\infty \left(e^{-\frac{t^3}3 + xt} +\sin\left(\tfrac {t^3}3 + xt\right)\right)dt.
\end{split}\end{equation}
Then, by \cite[Section 10.20]{dlmf} and \cite[Sections 9.3.35--46]{abst}, we have
\begin{equation}\label{e:jyaibi}\begin{split}
J_\nu(\nu z) &\asymp \left(\frac \zeta {1-z^2}\right)^{1/4} \left(\nu^{-1/3} \Ai(\nu^{2/3}\zeta) + \nu^{-5/3} B_0(\zeta)\Ai'(\nu^{2/3}\zeta) \right),\\
-zJ'_\nu(\nu z) &\asymp \left(\frac {1-z^2} \zeta \right)^{1/4} \left(\nu^{-2/3} \Ai'(\nu^{2/3}\zeta) + \nu^{-4/3} C_0(\zeta)\Ai(\nu^{2/3}\zeta)\right),\\
-Y_\nu(\nu z) &\asymp   \left(\frac \zeta {1-z^2}\right)^{1/4} \left(\nu^{-1/3} \Bi(\nu^{2/3}\zeta) + \nu^{-5/3} B_0(\zeta)\Bi'(\nu^{2/3}\zeta) \right),\\
zY'_\nu(\nu z) &\asymp   \left(\frac {1-z^2} \zeta \right)^{1/4} \left(\nu^{-2/3} \Bi'(\nu^{2/3}\zeta) + \nu^{-4/3} C_0(\zeta)\Bi(\nu^{2/3}\zeta)\right),
\end{split}\end{equation}
uniformly for $\nu \gg 1$ and $z >0$, where
$B_0$ and $C_0$ are positive smooth functions of $
\zeta$ obeying
\[
 B_0(\zeta) \asymp 
 \begin{cases} \zeta ^{-1/2}, \qquad &\zeta \ge 1, \\ \zeta ^{-2}, \qquad &\zeta \le -1,
 \end{cases} 
 \qquad C_0(\zeta) \asymp 
 \begin{cases}  \zeta ^{1/2} \qquad &\zeta \ge 1, \\ - \zeta ^{-1}, \qquad &\zeta \le -1. \end{cases}
\]

These bounds become simpler when $z$ is small. More specifically, by \cite[Section 9.7]{dlmf}, we have
\begin{equation}\label{e:airypos}\begin{array}{rclrcl}
\Ai(x) &\asymp& x^{-1/4} \exp(- \tfrac 23 x^{3/2}), \qquad & - \Ai'(x) &\asymp& x^{1/4} \exp(- \tfrac 23 x^{3/2}), \\
\Bi(x) &\asymp& x^{-1/4} \exp( \tfrac 23 x^{3/2}), & \qquad  \Bi'(x) &\asymp& x^{1/4} \exp( \tfrac 23 x^{3/2}),
\end{array}\end{equation}
for $x \gg 1$. Hence, for any $z_0 \in (0,1)$, we have
\begin{equation}\label{e:jyzsmall}\begin{array}{rclrcl}
J_\nu(\nu z) &\asymp&  \nu^{-1/2} e^{- \nu \xi},  \qquad & J'_\nu(\nu z) &\asymp& z^{-1}J_\nu(\nu z),\\
-Y_\nu(\nu z) &\asymp& \nu^{-1/2} e^{\nu \xi}, & \qquad Y'_\nu(\nu z) &\asymp& -z^{-1}Y_\nu(\nu z),
\end{array}\end{equation}
uniformly for $z \in (0,z_0]$ and $\nu \gg 1$, where 
\begin{equation}\label{e:xidef}
\xi = \int_z^1 t^{-1} (1-t^2)^{1/2}dt.
\end{equation}
 
The bounds also become simpler when $z$ is large. More specifically, by by \cite[Section 9.7]{dlmf}, we have
\begin{equation}\label{e:airyneg}\begin{split}
\Ai(-x) &\asymp x^{-1/4}\left( \cos( \tfrac 23 x^{3/2} - \tfrac \pi 4) + \tfrac 5 {48} x^{-3/2} \sin( \tfrac 23 x^{3/2} - \tfrac \pi 4)\right), \\
- \Ai'(-x) &\asymp x^{1/4} \left( \sin( \tfrac 23 x^{3/2} - \tfrac \pi 4) - \tfrac 7 {48} x^{-3/2} \cos( \tfrac 23 x^{3/2} - \tfrac \pi 4)\right), \\
\Bi(-x) &\asymp x^{-1/4}\left( -\sin( \tfrac 23 x^{3/2} - \tfrac \pi 4) + \tfrac 5 {48} x^{-3/2} \cos( \tfrac 23 x^{3/2} - \tfrac \pi 4)\right), \\
- \Bi'(-x) &\asymp x^{1/4} \left( \cos( \tfrac 23 x^{3/2} - \tfrac \pi 4) + \tfrac 7 {48} x^{-3/2} \sin( \tfrac 23 x^{3/2} - \tfrac \pi 4)\right), 
\end{split}\end{equation}
for $x \gg 1$. In particular, combining \eqref{e:airypos} and \eqref{e:airyneg} gives
\[
\exp( \tfrac 23 x_+^{3/2}) |\Ai(x)| + \exp(- \tfrac 23 x_+^{3/2})|\Bi(x)| \lesssim \langle x \rangle^{-1/4},
\]
\[
\exp( \tfrac 23 x_+^{3/2}) |\Ai'(x)| + \exp( -\tfrac 23 x_+^{3/2})|\Bi'(x)| \lesssim \langle x \rangle^{1/4}, 
\]
uniformly for $x \in \mathbb R$, where $x_+ = \max(x,0)$. Hence
\begin{equation}\label{e:jyzbigbound}
\begin{split}
 |J_\nu(\nu z)| &\lesssim \left|\frac \zeta {1-z^2}\right|^{1/4} e^{-\nu \xi_+}\left(\nu^{-1/2} \langle \zeta \rangle^{-1/4} + \nu^{-3/2} \langle \zeta \rangle^{-1/4} \right) \lesssim \nu^{-1/2} \langle z \rangle^{-1/2}e^{-\nu \xi_+}, \\
|Y_\nu(\nu z)| & \lesssim \left|\frac \zeta {1-z^2}\right|^{1/4}  e^{\nu \xi_+}\left(\nu^{-1/2} \langle \zeta \rangle^{-1/4} + \nu^{-3/2} \langle \zeta \rangle^{-1/4 } \right) \lesssim \nu^{-1/2} \langle z \rangle^{-1/2}e^{\nu \xi_+},
\end{split}
\end{equation}
uniformly for $\nu \gg 1$ and $z >0$, where  $\xi_+ = \xi$ as given by \eqref{e:xidef} when $z<1$, and $\xi_+ = 0$ when $z \ge 1$.

\section{Properties of $u_0$} \label{series construction u 0}

In this Appendix, we prove Lemma \ref{l:u0}, following \cite[Chapter 4, Section 1.1]{ya10}. Recall our notation from Section \ref{reduction section},
\begin{equation*}
\begin{gathered}
 \varphi_J(r) = r^{1/2} J_\nu(\lambda r), \qquad   \varphi_Y(r) = r^{1/2} Y_\nu(\lambda r), \\
 \nu = h^{-1}\big(m+\frac{h^2}{4} \big)^{1/2}, \quad m \ge -\frac{h^2}{4}, \quad \lambda = \frac{\sqrt{E+i\varepsilon}}{h}, \quad E \in [E_{\min}, E_{\max}], \quad \ep \ge 0.
 \end{gathered}
\end{equation*}
Then put $\varphi_0 = \varphi_J$ and define inductively
\begin{equation} \label{induct defn u0}
 \varphi_{n+1}(r) =  \frac{\pi}{2h^2} \int_0^r (\varphi_Y(r)\varphi_J(r') - \varphi_J(r)\varphi_Y(r')) V(r')\varphi_n(r')dr'.
\end{equation}
We shall prove suitable estimates on the $\varphi_n$ to be able to put
\begin{equation}
\label{e:u0}
 u_0(r) = \sum_{n=0}^\infty \varphi_n(r),
\end{equation}
with the series and its first derivative converging uniformly as $(r, E, \ep)$ vary in compact subsets of $\R_+ \times [E_{\min}, E_{\max}] \times  [0,\infty)$. 

\begin{proof}[Proof of Lemma \ref{l:u0}]
To prove the lemma, we recall several estimates on $\varphi_Y$ and $\varphi_J$.  Given $r^*>0$, by \cite[Sections, 10.6 (ii), 10.7(i)]{dlmf} there is $C_0$ depending continuously on $r^* > 0$, $\nu\geq 0$,  $\delta>0$, and $\lambda \neq 0$ such that
\begin{equation}
\label{e:basicEst}
\begin{gathered} |\varphi_J(r)|r^{-\nu-\frac 12} + |\varphi_Y(r)|r^{\nu +\delta-\frac 12}\le C_0\\
|\varphi'_J(r)|r^{-\nu+\frac 12} + |\varphi'_Y(r)|r^{\nu+\delta+\frac 12}\le C_0\\
 |\varphi''_J(r)|r^{-\nu+\frac 32} + |\varphi''_Y(r)|r^{\nu+\delta+\frac 32}\le C_0
\end{gathered} , \qquad r \in (0,r^*].
\end{equation}
Next,
\begin{equation}
|\varphi_0(r)| \le C_0r^{\nu + \frac 12},\qquad |\varphi_0'(r)| \le C_0r^{\nu - \frac 12},\qquad |\varphi_0''(r)| \le C_0r^{\nu - \frac 32}, \qquad r\in (0,r^*].
\end{equation}
To see that the series~\eqref{e:u0} converges uniformly on compact sets, we claim that
\begin{equation}
\label{varphiEst}
 |\varphi_n(r)| \le C_0 \cdots C_n r^{\nu+ n(2 -\delta)+ \frac 12}, \qquad r\in (0,r^*].
\end{equation}
where
\begin{equation} \label{C n}
 C_n =  C_0^2 \pi \sup |V| \frac{2\nu + (2n-1)(2-\delta)+2}{2 h^2 (2-\delta)n(2\nu+(n-1)(2-\delta)+2)}.
\end{equation}
Once we have~\eqref{varphiEst}, since $\lim C_n =0$, the Weierstrass and ratio tests shows that the convergence of \eqref{e:u0} is uniform as $(r, E, \ep)$ vary in compact subsets of $\R_+ \times [E_{\min}, E_{\max}] \times  [0,\infty)$. Moreover, since $\varphi_0\sim r^{\nu+\frac{1}{2}}$ as $r \to 0$, $u_0 > 0$ near zero.

To obtain~\eqref{varphiEst}, we start with $n=1$,
\[\begin{split}
\frac{2 h^2}{C_0^3 \pi \sup|V|}&|\varphi_1(r)| \\
&\le r^{-\nu -\delta+ \frac 12} \int_0^r (r')^{\nu+\frac 12} (r')^{\nu+\frac 12} dr' + r^{\nu + \frac 12} \int_0^r (r')^{-\nu-\delta+\frac 12} (r')^{\nu+\frac 12} dr'\\
&= (2\nu+2)^{-1} r^{\nu + \frac 52-\delta} + (2-\delta)^{-1} r^{\nu + \frac 52-\delta} = \frac {2\nu+4-\delta}{(2\nu+2)(2-\delta)} r^{\nu + \frac 52-\delta}, \qquad r\in (0,r^*],
\end{split}\]
or
\[
 |\varphi_1(r)| \le C_0C_1 r^{\nu+\frac 52-\delta}, \qquad r\in (0,r^*],
\]
where $C_1 = C_0^2\sup|V| (2\nu+4-\delta)/(2 \pi h^2 (2\nu+2)(2-\delta))$. 

Now, suppose the claim~\eqref{varphiEst} holds with $n=k$ for some $k\geq 1$, with $C_k$ given by \eqref{C n}. Then,
\[\begin{split}
&\frac{2 h^2}{C_0C_1C_2\dots C_{k}C_0^2 \pi \sup|V|}|\varphi_{k+1}(r)| \\
&\le r^{-\nu + \frac 12-\delta} \int_0^r (r')^{\nu+\frac 12} (r')^{\nu+k(2-\delta)+\frac 12} dr' + r^{\nu + \frac 12} \int_0^r (r')^{-\nu+\frac 12-\delta} (r')^{\nu+k(2-\delta)+\frac 12} dr'\\
&= (2\nu+k(2-\delta)+2)^{-1} r^{\nu +2(k+1)+ \frac 12} + ((k+1)(2-\delta))^{-1} r^{\nu +(k+1)(2-\delta)+ \frac 12}, \qquad r\in (0,r^*], 
\end{split}\]
or
\[
 |\varphi_{k+1}(r)| \le C_0 C_1\dots C_{k+1} r^{\nu +(k+1)(2-\delta)+ \frac 12}, \qquad r\in (0,r^*],
\]
where 
\begin{equation*}
C_{k+1} = C_0^2  \pi \sup |V| \frac{2\nu + (2k+1)(2-\delta)+2}{2 h^2 (2-\delta)(k+1)(2\nu+k(2-\delta)+2)}.
\end{equation*}
  In particular, \eqref{varphiEst} holds by induction.

To see that the first derivative of~\eqref{e:u0} converges uniformly on compact sets, observe that for $n\geq 0$,
\begin{equation} \label{varphiprimeEst}
\begin{split}
\frac{|\varphi_{n+1}'(r)|}{\sup |V|} &= \frac{\pi}{2 \sup |V| h^2}\Big| \int_0^r (\varphi_Y'(r)\varphi_J(r') - \varphi_J'(r)\varphi_Y(r')) V(r')\varphi_n(r')dr'\Big|\\
&\leq (2\pi h^2)^{-1}C_0\dots C_n C_0^2\Big(r^{-\nu-\frac{1}{2}-\delta}\int_0^r (r')^{2\nu+n(2-\delta)+1})dr'+r^{\nu-\frac{1}{2}}\int_0^r (r')^{n(2-\delta)+1-\delta}dr'\Big)\\
&\leq (2\pi h^2)^{-1}C_0\dots C_n C_0^2r^{\nu+n(2-\delta)+\frac{3}{2}-\delta}\Big(\frac{1}{2\nu+n(2-\delta)+2}+\frac{1}{(n+1)(2-\delta)}\Big)\\
&\leq C_0\cdots C_{n+1}r^{\nu+(n+1)(2-\delta)-\frac{1}{2}}, \qquad r\in (0,r^*].
\end{split}
\end{equation}
Now \eqref{u0 integral eqn}  and \eqref{u0 solves eqn} follow from \eqref{varphiEst} and \eqref{varphiprimeEst} by summing \eqref{induct defn u0} from $n = 0$ to $n = N$ and sending $N \to \infty$.

Next, to see that ~\eqref{uAsymptotic} holds, observe that for any $\delta>0$,
$$
\sum_{n=1}^\infty \varphi_n(r) \leq C_{\nu,\delta}r^{\nu+\frac{5}{2}-\delta},\qquad \sum_{n=1}^\infty \varphi_n'(r)\leq C_{\nu,\delta}r^{\tfrac{3}{2}-\delta}.
$$
Therefore, 
\begin{align*}
\lim_{r\to 0^+}\big( \tfrac{n-1}{2r}  u_0(r) - u_0'(r) \big)r^{-\nu+1/2}&=\lim_{r\to 0^+}\big( \tfrac{n-1}{2r} \varphi_0(r) - \varphi_0'(r) \big)r^{-\nu+1/2}\\
&=\lim_{r\to 0^+}\big( \tfrac{n-2}{2}J_\nu(\lambda r)r^{-1/2} - \lambda J_\nu'(\lambda r) r^{1/2} \big)r^{-\nu+1/2}\\
&=\frac{\lambda^\nu}{2^\nu\Gamma(\nu+1)}\big(\frac{n-2}{2}- \nu\big).
\end{align*}
Since $\nu\geq 0$,~\eqref{uAsymptotic} holds.\\
\end{proof}

Now that we have Lemma~\ref{l:u0}, we show that $u_0$ has useful monotonicity properties when $m$ is large enough. Recall the defining property \eqref{M 0 condition} of $M_0$ and consider $m\ge M_0$. 

\begin{lemma}
\label{nonnegativity}
Let $M_0$ as in~\eqref{M 0 condition} and $R_1$ as in~\eqref{e:r1}. Then for $m \ge M_0$, $u_0(r), u_0'(r)\geq 0$ for $r \in (0,R_0]$. 
\end{lemma} 
\begin{proof}
First, by Lemma~\ref{l:u0}, we have
 $$u''_0 = -h^{-2}(E - V - mr^{-2})u_0.$$
 Using $m \ge M_0$ and the definition \eqref{M 0 condition}, $E - V - mr^{-2} \leq 0$ on $(0, R_0]$. Hence $u_0(r)u_0''(r)\geq 0$ on $(0, R_0]$. Since also $u_0 (0) = 0$ and $u_0(r) > 0$ for $r>0$ small enough, the proof is completed by the following lemma. \\
 \end{proof}
\begin{lemma}
Suppose $f, f' \in AC_{\loc}(a,b)$, $f \in C([a,b];\mathbb{R})$, $f(a) = 0$, $f' \in C(a,b]$, and $f''(t)f(t)\geq 0$, for almost every in $t\in(a,b)$ in the sense of Lebesgue measure. Then $f$ has a fixed sign in $[a,b]$ and $f(b)$, $f'(b)$ have the same sign. 
\end{lemma}
\begin{proof}
If $f$ is identically 0 on $[a,b]$, then the claim is trivially true. Therefore, we will assume $f$ is not identically 0.
Assume $f$ attains its extremum $L \neq 0$ at an interior point of $(a,b)$. Replacing $f$ by $-f$ if necessary, we may assume $L > 0$. Set 
\begin{equation*}
x^* = \inf \{x \in (a,b) : f(x) = L \}.
\end{equation*}
Then $f > 0$ near $x^*$, hence $f'' \ge 0$ almost everywhere, in a neighborhood of $x^*$. Also $f'(x^*) = 0$. Now, for $x$ sufficiently close to but less than $x^*$,
\begin{equation*}
-f'(x) = \int^{x^*}_x f''(s) ds \ge 0 \implies f'(x) \le 0.
\end{equation*}

But then, by the mean value theorem, for $a < x < x^*$ and some $\overline{x} \in (x, x^*)$, 
\begin{equation*}
0 < f(x^*) - f(x) = f'(\overline{x})(x^* - x) \le 0,
\end{equation*}
a contradiction. 

So the extrema of $f$ must occur at the endpoints: $f(b)$ is the maximum or minimum of $f$ on $[a,b]$, and $f(a) = 0$ is the minimum or maximum, respectively. So $f$ has a fixed sign, hence also $f''$ has this sign almost everywhere.

Using the mean value theorem again, 
\begin{equation*}
\frac{f(b) - f(a)}{b - a} = \frac{f(b)}{b - a} = f'(\overline{x}), \qquad \text{some $\overline{x} \in (a,b).$}
\end{equation*}
 So $f(b)$ and $f'(\overline{x})$ have the same sign, so 
 \begin{equation*}
 f'(b) = f'(\overline{x}) + \int_{\overline{x}}^b f''(s) ds
 \end{equation*}
has the same sign as $f(b)$ as well.\\
\end{proof}

\section{Consequences of Meshkov's example for resolvent estimates}
\label{s:mesh}
Recall from Fredholm theory \cite[Theorems 3.8 and equation 6.0.2]{dz}  that, for $V \in L^\infty_{\text{comp}}(\R^n ; \C)$, the cutoff resolvent  $\chi (-h^2 \Delta + V -z)^{-1} \chi$, 
$\chi \in C_0^\infty(\R^n)$, is meromorphic in $(E_{\min}/2, 2E_{\max}) + i (0, \infty)$ and continues meromorphically to $(E_{\min}/2, 2E_{\max}) + i \R$.

To see that the optimal power in a resolvent estimate for $L^\infty$, compactly supported, complex valued potentials is $h^{-4/3}$, recall that~\cite{me92} constructs a \emph{complex-valued} potential $V\in L^\infty(\mathbb{R}^2)$ and a function $u$ such that 
$$
(-\Delta+V)u=0,\quad\text{in }\mathbb{R}^2,\qquad |u(x)|\leq C\exp(-C|x|^{4/3}),\qquad \|u\|_{L^2(\R^n)}=1.
$$
Changing variables, $y=hx$, and putting $w(y)=u(h^{-1}y)$, $V_h(y)=V(h^{-1}y)$, we have \\ $\|V_h\|_{L^\infty}\leq C<\infty$, 
$$
(-h^2\Delta +V_h)w=0,\qquad |w(y)|\leq C\exp(-C|y|h^{-4/3}).
$$
Let $\chi \in C_0^\infty(B(0,3))$ with $\chi \equiv 1$ near $B(0,2)$ and $\tilde{\chi}\in C_0^\infty(B(0,3))$ with $\tilde{\chi}\equiv 1$ on $\supp \chi$. Then, for $E\in [E_{\min}, E_{\max}]$,
\begin{align*}
\|\tilde{\chi}(-h^2\Delta& + (V_h+E)-E)(\chi w)\|_{L^2(\R^n)}=\|[-h^2\Delta,\chi]w\|_{L^2(\R^n)}\\
&\leq Ch\|w\|_{H_h^1(B(0,3)\setminus B(0,2))}\\
&\leq Ch(\|-h^2\Delta w\|_{L^2(B(0,4)\setminus B(0,1))}+\|w\|_{L^2(B(0,4)\setminus B(0,1))})\\
&\leq Ch(\|V_hw\|_{L^2(B(0,4)\setminus B(0,1))}+\|w\|_{L^2(B(0,4)\setminus B(0,1))})\leq Che^{-Ch^{-4/3}},
\end{align*}
For $h$ small enough,
$$\|\chi w\|_{L^2}\geq h-h\|u\|_{L^2(\mathbb{R}^2\setminus B(0,h^{-1})}\geq \tfrac{h}{2}.$$
Therefore, if the cutoff resolvent  $\tilde{\chi} (-h^2\Delta+\tilde{\chi}(V_h+E)-E-i0)^{-1}\tilde{\chi}$ exists, it must have 
$$
\|\tilde{\chi}(-h^2\Delta+\tilde{\chi}(V_h-E)-E-i0)^{-1}\tilde{\chi}\|_{L^2(\R^n) \to L^2(\R^n) }\geq C\exp(Ch^{-4/3}).
$$

\end{document}